\documentclass[11pt]{amsart}
\usepackage{latexsym,amssymb,amsmath,hyperref,empheq}
\usepackage{amsthm, dsfont}
\usepackage{tikz-cd}
\usepackage[all]{xy}
\usepackage[short,nodayofweek]{datetime}
\usepackage[active]{srcltx}

\leftmargin  \leftmargini
\def\NN{{\mathbb N}}
\def\ZZ{{\mathbb Z}}
\def\QQ{{\mathbb Q}}

\def\CC{{\mathbb C}}

\def\N{{\mathcal N}}

\def\gl{{\mathfrak{gl}}}

\def\ell{l}

\def\m{{\mathfrak m}}
\def\n{{\mathfrak n}}

\DeclareMathOperator{\GL}{GL}

\theoremstyle{plain}
\newtheorem{thm}{Theorem}[section]
\newtheorem{cor}[thm]{Corollary}
\newtheorem{lem}[thm]{Lemma}
\newtheorem{prop}[thm]{Proposition}

\theoremstyle{definition}
\newtheorem{defn}[thm]{Definition}
\newtheorem{remark}[thm]{Remark}
\newtheorem{eg}[thm]{Example}

\theoremstyle{remark}

\newtheoremstyle{Acknowledgements}
  {}
    {}
     {}
     {}
    {\bfseries}
    {}
     {.5em}
     {\thmname{#1}\thmnumber{ }\thmnote{ (#3)}}
\theoremstyle{Acknowledgements}

\date{\today, \currenttime} 

\begin{document}
\title[ ]{Eulerian summation operators and a remarkable family of polynomials }

\author{Kathrin Maurischat and Rainer Weissauer}
\address{\rm {\bf Kathrin Maurischat}, Mathematisches Institut,
   Heidelberg University, Im Neuenheimer Feld 205, 69120 Heidelberg, Germany }
\curraddr{}
\email{\sf maurischat@mathi.uni-heidelberg.de}
\address{\rm {\bf Rainer Weissauer}, Mathematisches Institut,
   Heidelberg University, Im Neuenheimer Feld 205, 69120 Heidelberg, Germany }
\curraddr{}
\email{\sf weissauer@mathi.uni-heidelberg.de}

\begin{abstract}
We give several families of polynomials which are related by Eulerian summation operators.
They satisfy interesting combinatorial properties like being integer-valued  at integral points.
This involves nearby-symmetries and a recursion for the values at half-integral points.
We also obtain identities for super Catalan numbers.
\end{abstract}

\maketitle
\setcounter{tocdepth}{2}
\tableofcontents
\section{Introduction}
Define the function $A:\NN\times \NN\to\NN$ by $A(k,l)=a(k,l)^2$, where
\begin{equation*} 
 a(k,l)\:=\:\sum_{\nu=0}^k\binom{l}{\nu}\binom{l-1+k-\nu}{l-1}\:.
\end{equation*}
In this paper we study numbers $P(m,n)$ which satisfy the summation equations
\begin{equation}\label{first_summation_equation}
 P(m,n+1)+2\cdot P(m,n)+P(m,n-1)\:=\: A(n,m)\footnote{The switching of $m$ and $n$ on the right hand side is due to other equalities explained later on. }
\end{equation}
for all $m\in\NN=\{1,2,3,\dots\}$ and all $n\in\{2,3,\dots\}$.  
It is worthwhile noting that once having found one solution $P(m,n)$ of these equations, any other solution $G(m,n)$ is given by $G(m,n)=P(m,n)+(-1)^n( c_1(m)n+c_0(m))$, 
where $c_0(m)$ and $c_1(m)$ are complex numbers depending only on $m$. Because, the difference $B(m,n)=G(m,n)-P(m,n)$ clearly satisfies the trivial summation equations
\begin{equation*}
 B(m,n+1)+2\cdot B(m,n)+B(m,n-1)\:=\: 0\:,
\end{equation*}
which exactly have the solutions $B(m,n)=(-1)^n(c_1n+c_0)$ with $c_1,c_0\in\CC$ for each single $m$. 
Notice further that $A(n,m)$ for fixed $m$ is polynomial in $n$ of degree $2(m-1)$. The summation operator
\begin{equation*}
 Sf(x)\:=\:f(x+\frac{1}{2})+f(x-\frac{1}{2})
\end{equation*}
is bijective on the polynomial ring $\CC[x]$. Hence there exists exactly one family of polynomials, which by abuse of notation we call $P(m,x)$, such that the polynomials
\begin{equation*}
 S^2 P(m,x)\:=\:  P(m,x+1)+2\cdot P(m,x)+P(m,x-1)
\end{equation*}
at each $x=n\in\NN$ have values $S^2P(m,n)=A(n,m)$.
From now on we will denote by $P(m,n)$ the special solution of (\ref{first_summation_equation}) given by the values of these polynomials $P(m,x)$ at places $x=n\in\NN$.
This solution has a number of interesting properties of which we collect the two most important ones here. 
First, it describes an integer-valued function
\begin{equation*}
 P:\NN\times\NN\:\to\:\ZZ\:,
\end{equation*}
and for $n\geq m$ the values $P(m,n)$ are indeed natural numbers. This is shown in Corollary~\ref{corrolar_1} and Remark~\ref{remarks}~(iii).
Second,
consider the summation equation in the first variable
\begin{equation}\label{second_summation_equation}
 \widehat P(m+1,n)+2\cdot \widehat P(m,n)+\widehat P(m-1,n)\:=\:A(m,n)\:.
\end{equation}
By virtue of formula (\ref{first_summation_equation}), the values $\widehat P(m,n)=P(n,m)+(-1)^m(c_1(n)m+c_0(n))$  give a solution of (\ref{second_summation_equation}) 
for all $c_0(n),c_1(n)\in\CC$.
In Theorem~\ref{thm_summation_operator} we show that indeed $P(m,n)$ also is a solution of (\ref{second_summation_equation}).
It follows that $P(m,n)$ is nearly symmetric for all $m,n\in\NN$,
\begin{equation*}
 P(m,n)\:=\:P(n,m)+(-1)^m\bigl(c_1(n)m+c_0(n)\bigr)\:,
\end{equation*}
and we show $c_1(n)=(-1)^{n-1}$ and $c_0(n)=(-1)^n\cdot n$.
Hence, $P(m,n)$ is almost a polynomial in the first variable, too.
For  the first values $P(m,n)$ see Table~1.\\

The strategy of this paper is reverse to the above exposition. 
In Proposition~\ref{definierende_eigenschaften} we  iteratively define a family of polynomials $P(m,x)$  imposing a number of properties on them.
Then we determine an explicit formula for these $P(m,x)$ and give the results on integer values and nearby symmetry in Section~\ref{sec_polnomials}.
In Theorem~\ref{thm_summation_operator} we define two polynomials $A_1(m,x)$ and $A_2(x,n)$ which both interpolate the values $A(m,n)$. 
Here $A_2(\cdot,m)$ is the polynomial interpretation of $A(\cdot,m)$  used above.
By  counting arguments we show  the first variable summation equation
\begin{equation*}
 P(m+1,x)+2P(m,x)+P(m-1,x)\:=\:A_1(m,x)\:.
\end{equation*}
This implies (\ref{second_summation_equation}) for all $x=n\in \NN$ and
 the second variable summation equation $P(m,x+1)+2P(m,x)+P(m,x-1)=A_2(x,m)$ for all $x=n\in\NN$. Both sides  being polynomials,  the equation must hold for all $x$.
As a consequence, we obtain 
 a polynomial identity $m^2A_1(m,x)=x^2A_2(x,m)$  (Proposition~\ref{properties_A1_A2}), respectively the symmetry 
\begin{equation*}
 m^2A(m,n)\:=\:n^2A(n,m)\:.
\end{equation*}

The numbers $P(m,n)$ arise as dimensions of certain $\GL(n\vert n)$-modules, where $\GL(n\vert n)$ is a  general linear super group. See \cite{rank-3-and-4} for details. 
To our surprise their fascinating combinatorial properties have not  been studied  in the literature so far. \\

The values $P(m,n)$ satisfy nice summation equations in the first and in the second variable.
One may ask whether this also  holds  for  the mixed summation equation. Define the family $Q(m,x)$ of polynomials  by the images of $P(m,x)$ under the summation operator
\begin{equation}\label{third_summation_equation}
 Q(m,x)\:=\:P(m,x)+P(m,x-1)+P(m-1,x)+P(m-1,x-1)\:.
\end{equation}
We study the combinatorial properties of the values $Q(m,n)$ in section~\ref{sec_mixed}. Like for the construction of $P(m,x)$, in Proposition~\ref{definierende_eigenschaften_mixed}
we impose properties on a family $Q(m,x)$ of polynomials defining them iteratively, and show that the right hand sides of  (\ref{third_summation_equation}) satisfy these properties.
For the summation operator $S$ it holds
\begin{equation*}
 Q(m,x+\frac{1}{2})\:=\:S\bigl(P(m,x)+P(m-1,x)\bigr)\:.
\end{equation*}
This suggests that the values at half-integral numbers of all the polynomials involved should allow a description. We give one by Proposition~\ref{prop_x=1/2_values_polynomials} and 
a recursion process. This also  justifies the mixing of the summation operators $S$ and $\widetilde S$
\begin{equation*}
 \widetilde Sf(x)\:=\:f(x+1)+f(x)\:.
\end{equation*}

We also obtain the identity
\begin{equation*}
 Q(m,x)\:=\:\widetilde E\bigl(A_2(x,m)+A_2(x,m-1)\bigr)\:,
\end{equation*}
where the Euler operator $\widetilde E=\widetilde S^{-1}$ is the inverse on polynomials of  the  operator $\widetilde S$.

We get a new polynomial identity  from this in Proposition~\ref{prop_new_polynomial_id}.
This involves the preimage $\widetilde E\bigl(A_2(x,m)\bigr)$ for which we 
%
have to compute the polynomials
\begin{equation*}
F(x,\nu,n)\:=\: \widetilde E\bigl(\begin{bmatrix}
                x\\n
               \end{bmatrix}\cdot\begin{bmatrix} x-\nu\\n\end{bmatrix}\bigr)\:,
\end{equation*}
where $\begin{bmatrix} x\\n \end{bmatrix}=\frac{1}{n!}x(x-1)\cdots(x-(n-1))$.
This is the purpose of Section~\ref{sec_Euler_operator}.
We give a method for finding the preimage $\widetilde E(f)$ for a polynomial $f$ in case that a series of subsequent values $f(0),f(1),\dots,f(deg f)$ is given.
This is a result parallel to Euler's summation formula for the solution $\tilde f$ of the difference operator $\tilde f(x+1)-\tilde f(x)=f(x)$ (see \cite[11.10]{koenigsberger}). 
But because the inverse of the difference operator is a discrete integration operator, whereas $\widetilde E$ is not, our formula in Proposition~\ref{urbild-polynom} is more bulky.
The constant coefficients $c(\nu,n-1)=F(0,\nu,n-1)$ of the above polynomials satisfy two recursion formulas themselves (Proposition~\ref{prop-preimages})
\begin{equation*}
 c(\nu,n)\:=\:c(\nu-2,n)+c(\nu-1,n-1)\:,
\end{equation*}
and 
\begin{equation*}
 c(\nu,n)\:=\:-c(\nu-2,n)+\frac{\nu}{n}c(\nu-1,n-1)\:.
\end{equation*}
They are  given by Gessel's~\cite{gessel} super Catalan numbers $C(m,k)=\frac{(2m)!(2k!)}{2\cdot m!k!(m+k)!}$ 
\begin{equation*}
 c(\nu,n)\:=\: \frac{(-1)\mu}{2^{2n}}\cdot C(n-\mu,\mu)\:,
\end{equation*}
if $\nu$ is of the form $\nu=n-2\mu$. If $\nu$ is not of this form, then $c(\nu,n)$ is zero.
In Corollary~\ref{cor-on-super-Catalan} we give some identities for super Catalan numbers which we obtain from the above construction.
\section{A family of polynomials}\label{sec_polnomials}
\begin{prop}\label{definierende_eigenschaften}
For $m=0,1,2,\dots$ there is a unique family of polynomials $P(m,x)$   in $\QQ[x]$ with the following properties.
\begin{itemize}
 \item [(i)] $P(0,x)=0$.
 \item [(ii)] $\deg_x P(m,x)\leq 2(m-1)$ for all $m>0$.
 \item [(iii)]  $P(m,x)=P(m,-x)$ holds for all  $m\in \NN_0$.
 \item [(iv)] The function $f(m,n)=P(m,n)+(-1)^{m+n}\cdot m$ is a symmetric function on $\NN\times \NN$, i.e. $f(m,n)=f(n,m)$.
\end{itemize}
\end{prop}
\begin{proof}[Proof of Proposition~\ref{definierende_eigenschaften}]
We show that the properties (i)--(iv)  uniquely define the polynomials $P(m,x)$ by recursion.
For $m=0$ the polynomial $P(0,x)=0$ is fixed by property (i).
For $m=1$, by (ii) we know $P(1,x)=c$ is a constant polynomial the constant $c$ being given by $P(1,0)=c$. By (iv) we see
\begin{equation*}
 P(1,0)+(-1)^{1+0}\cdot 1\:=\: P(0,1)+(-1)^{0+1}\cdot 0\:,
\end{equation*}
so $c=P(0,1)+1=1$. Assuming $P(k,x)$ to be constructed for $0\leq k\leq m$ we obtain by property (iv) the following values of $P(m+1,x)$
\begin{equation*}
 P(m+1,k)\:=\:P(k,m+1)+(-1)^{m+k}(m+1-k)\:.
\end{equation*}
Using (iii) we find $P(m+1,-k)=P(m+1,k)$ and we thus have fixed the values $P(m+1,x)$ at the $2m+1$ places $x\in\{-m,\dots,0,\dots,m\}$.
But by (ii) the degree of $P(m+1,x)$ is at most $2m$, hence $P(m+1,x)$ is the unique interpolation polynomial of degree $2m$ for the above values. 
\end{proof}
For example, condition (iv) together with (i) implies
\begin{equation*}
 P(m,0)\:=\:(-1)^{m-1}\cdot m\:,
\end{equation*}
as well as
\begin{equation*}
 P(m,1)\:=\:1+(-1)^m(m-1)\:.
\end{equation*}
In particular 
\begin{small}
\begin{align*}
 P(0,x)&=0\:,\\
 P(1,x)&=1\:,\\
 P(2,x)&=4x^2-2\:,\\
 P(3,x)&=4x^4-8x^2+3\:,\\
 P(4,x)&= \frac{16}{9}x^6 - \frac{56}{9}x^4 + \frac{112}{9}x^2 - 4\:,\\
 P(5,x)&=\frac{4}{9}x^8 - \frac{16}{9}x^6 + \frac{92}{9}x^4 - \frac{152}{9}x^2 + 5\:,\\
 P(6,x)&= \frac{16}{225}x^{10} - \frac{8}{45}x^8 + \frac{848}{225}x^6 - \frac{592}{45}x^4 + \frac{1612}{75}x^2 - 6\:,\\
 P(7,x)&= \frac{16}{2025}x^{12} + \frac{32}{2025}x^{10} + \frac{596}{675}x^8 - \frac{7984}{2025}x^6 + \frac{34696}{2025}x^4 - \frac{5872}{225}x^2 + 7\:,\\
 P(8,x)&=\frac{64}{99225}x^{14} + \frac{32}{4725}x^{12} + \frac{64}{405}x^{10} - \frac{46384}{99225}x^8 + \frac{27968}{4725}x^6 - \frac{41312}{2025}x^4 + \frac{339392}{11025}x^2 - 8\:.\\
\end{align*}
\end{small}
The proof of Proposition~\ref{definierende_eigenschaften} shows that the values $P(m,k)$ for all $k=-m,\dots,m$ are integers for all $m\geq 0$. Hence by (iv), 
for an integer $j>0$
the value
\begin{equation*}
 P(m,m+j)\:=\:P(m+j,m)+(-1)^j\cdot j
\end{equation*}
also is integral. This proves
\begin{cor}\label{corrolar_1}
The function $P\::\:\NN\times\NN \to \ZZ$
defined by the values $P(m,n)$ on natural numbers of the family
 of polynomials $P(m,x)$ in Proposition~\ref{definierende_eigenschaften} is integer-valued.
\end{cor}
 Let $m$ be a natural number. For  integers $0\leq\mu\leq m-1 $ put
\begin{equation*}
 \mu^\ast\:=\: m-1-\mu\:.
\end{equation*}
For integers $0\leq \nu,\mu\leq m-1$ we define the polynomials
\begin{align*}
 t(\nu,\mu,m;x)&=\prod_{k=1}^{\mu}(x+\nu-\mu+k)\cdot\prod_{l=1}^{\nu}(x-1-\mu+l)\\
 &=(x+\nu)\cdots(x+\nu-\mu+1)\cdot(x+\nu-\mu-1)\cdots(x-\mu)\:.
\end{align*}
\begin{table}
\caption{Initial values of the function $P:\NN\times\NN\to\ZZ$.}
\begin{small}
\begin{tabular}{c|c|c|c|c|c|c|c|c}
 $P(m,n)$&$1$&$2$&$3$&$4$&$5$&$6$&$7$&$8$\\
  \hline
  $1$&$1$&$1$&$1$&$1$&$1$&$1$&$1$&$1$\\
   \hline
  $2$&$2$&$14$&$34$&$62$&$98$&$142$&$194$&$254$\\
 \hline
 $3$&$-1$&$35$&$255$&$899$&$2303$&$4899$&$9215$&$15875$\\
 \hline 
 $4$&$4$&$60$&$900$&$5884$&$24196$&$75324$&$194820$&$441340$\\
  \hline
 $5$&$-3$&$101$&$2301$&$24197$&$151805$&$676197$&$2376701$&$7031301$\\
 \hline
 $6$&$6$&$138$&$4902$&$75322$&$676198$&$4160778$& $19475142$&$74307834$\\
  \hline
 $7$&$-5$&$199$&$9211$&$194823$&$2376699$&$19475143$&$118493179$&$573785095$\\
 \hline
 $8$&$8$&$248$&$15880$&$441336$&$7031304$&$74307832$&$573785096$&$3465441272$
\end{tabular}
\end{small}
\end{table}
\begin{prop}\label{the_polynomials}
 The polynomials
 \begin{equation*}
  P(m,x)\:=\:\sum_{\nu,\mu=0}^{m-1}\frac{t(\nu,\mu,m;x)\cdot t(\mu^\ast,\nu^\ast,m;x)}{\nu!\nu^\ast !\mu!\mu^\ast !}
 \end{equation*}
for $m>0$, and $P(0,x)=0$ satisfy the properties of Proposition~\ref{definierende_eigenschaften}.
\end{prop}
\begin{proof}[Proof of Proposition~\ref{the_polynomials}]
 By definition, condition (i) of Proposition~\ref{definierende_eigenschaften} is satisfied.
 For the summands of $P(m,x)$ we have for all $\nu,\mu$
 \begin{equation*}
  \deg_x t(\nu,\mu,m;x)\cdot t(\mu^\ast,\nu^\ast,m;x)\:=\:\nu+\mu+\nu^\ast+\mu^\ast\:=\: 2(m-1)\:,
 \end{equation*}
so the same holds true for $P(m,x)$. Hence property (ii) holds.
Obviously,
\begin{equation*}
 t(\nu,\mu,m;-x)\:=\:(-1)^{\nu+\mu}\cdot t(\mu,\nu,m;x)\:,
\end{equation*}
so condition (iii) follows
\begin{align*}
P(m,-x)
&= \sum_{\mu,\nu=0}^{m-1}(-1)^{2(m-1)}\frac{t(\mu,\nu,m;x)\cdot t(\nu^\ast,\mu^\ast,m;x)}{\nu!\nu^\ast !\mu!\mu^\ast !}\:=\:P(m,x)\:.\\
\end{align*}
In order to prove (iv), which is trivial for $m=n$, we assume  $n> m$ without loss of generality. Substituting $\mu\mapsto m-1-\mu$ we may write
\begin{equation*}
 P(m,n)\:=\: \sum_{\nu^\ast,\mu\ast=0}^{m-1}\frac{t(\nu,\mu^\ast,m;n)\cdot t(\mu,\nu^\ast,m;n)}{\nu!\nu^\ast !\mu!\mu^\ast !}\:.
\end{equation*}
Notice that for $n\leq \mu^\ast$ the value
\begin{equation*}
 t(\nu,\mu^\ast,m;n)\:=\:(n+\nu)\cdots(n+\nu-\mu^\ast+1)\cdot(n+\nu-\mu^\ast-1)\cdots(n-\mu^\ast)
\end{equation*}
is zero unless $n+\nu-\mu^\ast=0$, where the value is $(-1)^\nu \nu!\mu^\ast!$. 
Similarly, $ t(\nu,\mu^\ast,m;n)$ is zero for $n\leq \nu^\ast$ unless $n+\nu-\mu^\ast=0$, in which case it is $(-1)^\mu \mu!\nu^\ast!$. So we obtain
\begin{equation*}
 P(m,n) \:=\:\sum_{\nu^\ast,\mu\ast=0}^{n-1}\frac{t(\nu,\mu^\ast,m;n)\cdot t(\mu,\nu^\ast,m;n)}{\nu!\nu^\ast !\mu!\mu^\ast !}\:+\:\sum_{\nu+\mu=m-1-n}(-1)^{\nu+\mu}\:.
\end{equation*}
In this expression, the second sum is $(-1)^{m+n-1}(m-n)$. Substituting $i=m-1-\nu$ and $j=m-1-\mu$  the first sum  becomes
\begin{equation*}
 \sum_{i,j=0}^{n-1}\frac{t(i,n-1-j,n;m)t(j,n-1-i,n;m)}{i!(n-1-i)!j!(n-1-j)!}\:=\: P(n,m)\:.
\end{equation*}
So for $n>m$
\begin{equation*}
 P(m,n)\:=\:P(n,m)+(-1)^{m+n-1}(m-n)\:.
\end{equation*}
Hence condition (iv) of Proposition~\ref{definierende_eigenschaften} holds for all integers $m,n>0$.
\end{proof}
\begin{defn}\label{Weyl-dimension} 
For integers $\alpha$ and $\beta$ define the natural number 
 \begin{equation*}
 D_n(\alpha+1,\beta)\:=\: \left\{\begin{array}{ll}\frac{n}{\alpha+\beta+1}\binom{n+\alpha}{\alpha}\binom{n-1}{\beta}&\textrm{ if } \alpha\geq 0\textrm{ and } 0\leq\beta\leq n-1\\
                                0& \textrm{else }
                               \end{array}\right..
\end{equation*}
\end{defn}
\begin{remark}\label{remarks}
i)
  We have seen $\deg_xP(m,x)=2(m-1)$. So property (ii) of Proposition~\ref{definierende_eigenschaften} can be sharpened as
\begin{itemize}
\item [(ii')] $\deg_x P(m,x)= 2(m-1)$ for all $m>0$.
\end{itemize}

ii)
Fixing the first variable,  the function $P(m,x)$ is polynomial in $x$ by definition.
By property (iv) of Proposition~\ref{definierende_eigenschaften} the values  $P(m,n)$  are nearly symmetric
\begin{equation*}
P(n,m)\:=\:P(m,n)+(-1)^{m+n}(m-n)\:,
\end{equation*}
Hence  for fixed $n\in\NN$ the function $P(m,n)$ is almost a polynomial  of degree $2(n-1)$ in the first variable $m$.

iii)
For integers $n>0$ there is an appealing presentation of $t(\nu,\mu,m;n)$
\begin{equation*}
 \frac{t(\nu,\mu,m;n)}{\nu!\mu!}\:=\:\left\{\begin{array}{ll}\frac{n}{n+\nu-\mu}\binom{n+\nu}{\nu}\binom{n-1}{\mu}& \textrm{if } n+\nu-\mu\not=0\\
                                             (-1)^\nu&\textrm{if } n+\nu-\mu=0
                                            \end{array}\right.\:.
\end{equation*}
In particular, for integers $n\geq m>0$ we obtain
\begin{equation*}
 P(m,n)\:=\:\sum_{\nu,\mu=0}^{m-1}\frac{n^2}{(n+\nu-\mu)^2}\binom{n+\nu}{\nu}\binom{n-1}{\mu}\binom{n+(m-1-\mu)}{m-1-\mu}\binom{n-1}{m-1-\nu}\:.
\end{equation*}
Equivalently, using Definition~\ref{Weyl-dimension}  for integers $n\geq m>0$
\begin{equation*}
 P(m,n)\:=\:\sum_{\nu,\mu=0}^{m-1}D_n\big(\nu+1,n-1-\mu)\cdot D_n((m-1-\mu)+1,n-1-(m-1-\nu)\bigr)\:.
\end{equation*}
As of two natural numbers, any single summand of $P(m,n)$ is a natural number for the integers $n\geq m>0$. 
Hence 
the values $P(m,n)$ are natural numbers for all integers $n\geq m$.
In general,   for integers $m,n>0$ define the numbers
\begin{equation*}
 \widetilde P(m,n)\:=\!\!\!\!\sum_{\nu,\mu=0}^{\min\{n-1,m-1\}}\!\!\!\!\!\!\!\!D_n\big(\nu+1,n-1-\mu)\cdot D_n((m-1-\mu)+1,n-1-(m-1-\nu)\bigr)\:.
\end{equation*}
Hence   $P(m,n)=\widetilde P(m,n)$ holds for $n\geq m>0$, whereas for $n<m$ we obtain
\begin{equation*}
 P(m,n)\:=\: \widetilde P(m,n) +(-1)^{m+n-1}(m-n)\:.
\end{equation*}
On the other hand, we know $P(m,n)=P(n,m)+(-1)^{m+n-1}(m-n)$ by property (iv). It follows for $n<m$
\begin{equation*}
 \widetilde P(m,n)\:=\: P(n,m)\:=\: \widetilde P(n,m)\:.
\end{equation*}
Hence $\widetilde P(m,n)$ is symmetric.
\end{remark}

\section{Summation operators}\label{sec_summation_op}
   We define the summation operator $S$ acting on a function $f$ 
 \begin{equation*}
  S f(x)\:=\: f(x+\frac{1}{2})+f(x-\frac{1}{2})\:.
 \end{equation*}
 On polynomials $S$ acts by $ S(x^n)=\sum_{k=0}^n\binom{n}{k}x^{n-k}2^{-k}(1+(-1)^k)$. The  preimages $S^{-1}(x^n)$ can be obtained uniquely by recursion starting with $S^{-1}(0)=0$
 and $S^{-1}(1)=\frac{1}{2}$.
 Hence   $S$ is bijective on polynomial rings over fields of characteristic $\neq 2$.
 Noticing
 \begin{equation*}
  S^2f(x)\:=\:f(x+1)+2\cdot f(x)+f(x-1)\:,
 \end{equation*}
there is a natural understanding of the definition of the summation operator $S^2$ on functions $f$ on the integers, i.e. on sequences. 

Let denote 
 \begin{equation*}
  \begin{bmatrix}
   x\\k
  \end{bmatrix}\:=\:\frac{x(x-1)\cdots (x-k+1)}{k!}
 \end{equation*}
the polynomial of degree $k$ interpolating  the binomial coefficient $\begin{pmatrix}
   n\\k
  \end{pmatrix}=\begin{bmatrix}
   n\\k
  \end{bmatrix}$ for integers $n\geq 0$. In particular, $\begin{bmatrix}   n\\0  \end{bmatrix}=1$ is the constant polynomial.
  
   The following theorem shows that the images of the summation operator $S^2$ applied to the nearly symmetric function $P(m,n)$ of Corollary~\ref{corrolar_1}, 
   once in the first and once in the second variable, behave similarly.
\begin{thm}\label{thm_summation_operator}
Let $P:\NN\times\NN\to\ZZ$ be the function defined in Corollary~\ref{corrolar_1}.
Let $S_1$ and $S_2$ be the summation operators  in the first and second variable, respectively, and consider the function $A:\NN\times\NN\to\NN$ given by
\begin{equation}\label{summation_equation_to_show_discrete}
  S_1^2 P(m,n) \:=\: P(m+1,n)+2\cdot P(m,n)+P(m-1,n)\:=\: A(m,n)\:.
  \end{equation}
  Then it holds 
 \begin{equation}\label{summation_equation_2_to_show_discrete}
  S_2^2 P(m,n) \:=\: P(m,n+1)+2\cdot P(m,n)+P(m,n-1)\:=\: A(n,m)\:.
  \end{equation} 
For all integers $m=1,2,\dots$ define  polynomials of degree $2m$
\begin{equation*}
 A_1(m,x)\:=\:  \left(\sum_{\nu=0}^m \begin{bmatrix}
   x-1+\nu\\\nu
  \end{bmatrix}
  \begin{bmatrix}
   x\\m-\nu
  \end{bmatrix}
\right)^2\:.
\end{equation*}
For all integers $n=1,2,\dots$ define polynomials of degree $2(n-1)$
\begin{equation*}
 A_2(x,n)\:=\:  \left(\sum_{\nu=0}^n \begin{pmatrix}
   n\\\nu
  \end{pmatrix}\begin{bmatrix}
   x-\nu+n-1\\n-1
  \end{bmatrix}
\right)^2\:.
\end{equation*}
Then for all integers $n,m>0$ the following holds
\begin{equation*}
 A_1(m,n)\:=\: A(m,n)\:=\:A_2(m,n)\:.
\end{equation*}
Let $P(m,x)$ be the family  of polynomials of Proposition~\ref{definierende_eigenschaften}  defining the numbers $P(m,n)$.
Then for all integers $m>0$ 
 \begin{equation}\label{summation_equation_to_show}
  S_1^2 P(m,x) \:=\: P(m+1,x)+2\cdot P(m,x)+P(m-1,x)\:=\: A_1(m,x)\:.
  \end{equation}
 Furthermore, for all integers $m>0$ we have
 \begin{equation}\label{summation_equation_2_to_show}
  S_2^2 P(n,x) \:=\: P(n,x+1)+2\cdot P(n,x)+P(n,x-1)\:=\: A_2(x,n)\:.
  \end{equation} 
\end{thm}
\begin{table}
\caption{Initial values of the function $A:\NN\times\NN\to\NN$.}
\begin{small}
\begin{tabular}{c|c|c|c|c|c|c|c|c}
 $A(m,n)$&$1$&$2$&$3$&$4$&$5$&$6$&$7$&$8$\\
  \hline
  $1$&$4$&$16$&$36$&$64$&$100$&$144$&$196$&$256$\\
   \hline
  $2$&$4$&$64$&$324$&$1024$&$2500$&$5184$&$9604$&$16384$\\
 \hline
 $3$&$4$&$144$&$1444$&$7744$&$28900$&$85264$&$213444$&$473344$\\
 \hline 
 $4$&$4$&$256$&$4356$&$36864$&$202500$&$831744$&$2775556$&$7929856$\\
  \hline
 $5$&$4$&$400$&$10404$&$129600$&$1004004$&$5588496$&$24423364$&$88811776$\\
 \hline
 $6$&$4$&$576$&$21316$&$369664$&$3880900$&$28472896$&$159820164$&$729432064$\\
  \hline
 $7$&$4$&$784$&$39204$&$906304$&$12460900$&$117418896$&$830246596$&$4687319296$\\
 \hline
 $8$&$4$&$1024$&$66564$&$1982464$&$34692100$&$410305536$&$3588728836$&$24706809856$
\end{tabular} 
\end{small}
\end{table}
Before we prove Theorem~\ref{thm_summation_operator} let us include the following remarks.
Consider the symmetric numbers $\widetilde P(m,n)$ defined in Remark~\ref{remarks}(iii). For $n>m$ they satisfy $\widetilde P(m+i,n)=P(m+i,n)$ for $i=-1,0,1$. Similarly for $n<m$, they satisfy
$\widetilde P(m+i,n)=P(n,m+i)=P(m+i,n)+(-1)^{m+i+n}(m+i-n)$ for $i=-1,0,1$ by Proposition~\ref{definierende_eigenschaften}~(iv).
For $m=n$ we obtain $\widetilde P(n+i,n)=P(n+i,n)$ for $i=-1,0$, but $\widetilde P(n+1,n)=P(n+1,n)-1$.
By Theorem~\ref{thm_summation_operator} it follows for all $m\not=n$
\begin{equation*}
 \widetilde P(m+1,n)+2\cdot\widetilde P(m,n)+\widetilde P(m-1,n)\:=\: A(m,n)\:,
\end{equation*}
whereas 
\begin{equation*}
 \widetilde P(n+1,n)+2\cdot\widetilde P(n,n)+\widetilde P(n-1,n)+1\:=\: A(n,n)\:.
\end{equation*}
This proves the following corollary.
\begin{cor}
 For all integers $m,n>0$ the symmetric numbers
\begin{equation*}
  \widetilde P(m,n)\:=\!\!\!\!\!\!\!\!\sum_{\nu,\mu=0}^{\min\{n-1,m-1\}}\!\!\!\!\!\!\!\!\frac{n^2}{(n+\nu-\mu)^2}\binom{n+\nu}{\nu}\binom{n-1}{\mu}\binom{n+m-1-\mu}{m-1-\mu}\binom{n-1}{m-1-\nu}
\end{equation*}
satisfy  the summation equation
\begin{equation*}
S_1^2\widetilde P(m,n)+\delta_{m,n}\:=\: A(m,n)\:.
\end{equation*}
\end{cor}
We include a list of polynomials $A_1(m,x)$ and $A_2(x,n)$,
\begin{small}
\begin{align*}
A_1(1,x)&= 4x^2\:,&A_2(x,1)&= 4\:,\\
A_1(2,x)&= 4x^4\:,&A_2(x,2)&= 16x^2\:,\\
A_1(3,x)&= \frac{16}{9}x^2(x^2+\frac{1}{2})^2\:,&A_2(x,3)&= 16(x^2+\frac{1}{2})^2\:,\\
A_1(4,x)&= \frac{4}{9}x^2(x^2+2)^2\:,&A_2(x,4)&= \frac{64}{9}x^2(x^2+2)^2\:,\\
A_1(5,x)&= \frac{16}{225}x^2(x^4+5x^2+\frac{3}{2})^2\:,&A_2(x,5)&= \frac{16}{9}(x^4+5x^2+\frac{3}{2})^2\:,\\
A_1(6,x) &= \frac{16}{2025}x^4\bigl(x^{4} + 10x^2 + \frac{23}{2}\bigr)^2\:,&A_2(x,6) &= \frac{64}{225}x^2\bigl(x^{4} + 10x^2 + \frac{23}{2}\bigr)^2\:,\\
A_1(7,x)&= \frac{64}{99225}x^2\bigl(x^6 + \frac{35}{2}x^4 + 49x^2 + \frac{45}{4}\bigr)^2\:,&A_2(x,7)&= \frac{64}{2025}\bigl(x^6 + \frac{35}{2}x^4 + 49x^2 + \frac{45}{4}\bigr)^2\:,\\
A_1(8,x)&=\frac{4}{99225}x^4(x^2+6)^2\bigl( x^{4} + 22x^{2} + 22\bigr)^2\:,&A_2(x,8)&=\frac{256}{99225}x^2(x^2+6)^2\bigl( x^{4} + 22x^{2} + 22\bigr)^2\:.
\end{align*}
\end{small}
The polynomials $A_1(m,x)$ and $A_2(x,n)$ satisfy the following properties.
\begin{prop}\label{properties_A1_A2}
 \begin{itemize}
  \item [(a)] For all $m>0$ there is an identity of polynomials
  \begin{equation*}
   m^2\cdot A_1(m,x)\:=\:x^2\cdot A_2(x,m)\:.
  \end{equation*}
  In particular, $m^2A(m,n)=n^2A(n,m)$ is a  symmetric function on $\NN^2$.
\item[(b)] The polynomial 
\begin{equation*}
 A_2(x,n)\:=\: \frac{2^{2n}}{(n-1)!^2}\cdot x^{2(n-1)}+\dots+\left(1+(-1)^{n-1}\right)^2
\end{equation*}
of degree $2(n-1)$ is even $A_2(x,n)=A_2(-x,n)$. Its value at $x=1$ is $A_2(1,n)=4n^2$.
 \end{itemize}
\end{prop}
\begin{proof}[Proof of Proposition~\ref{properties_A1_A2}]
The $\nu$-th summand of the sum in $A_2(x,m)$ is
\begin{equation*}
 \binom{m}{\nu}\begin{bmatrix}x-\nu+m-1\\m-1\end{bmatrix}\:=\:m\cdot\frac{(m-\nu+x-1)\cdots(1+x-1)}{(m-\nu)!}\frac{(x-1)\cdots(x-(\nu-1))}{\nu!}\:,
\end{equation*}
so
\begin{equation*}
 \frac{x}{m}\cdot \binom{m}{\nu}\begin{bmatrix}x-\nu+m-1\\m-1\end{bmatrix}\:=\:\begin{bmatrix}m-\nu+x-1\\m-\nu\end{bmatrix}\begin{bmatrix}x\\\nu\end{bmatrix}
\end{equation*}
is the $(m-\nu)$-th summand of the sum in $A_1(m,x)$. Part~(a) follows.

 Property~(iii) of the family $P(m,x)$ (cf. Proposition~\ref{definierende_eigenschaften}) of being even polynomials $P(m,x)=P(m,-x)$ is inherited by their images 
 $S_1^2P(m,x)=A_1(m,x)$ under the summation operator. By part (a) of this proposition, it carries over to $A_2(x,m)=A_2(-x,m)$.
 The leading term of the polynomial $\begin{bmatrix}
   x\\n-1
  \end{bmatrix}$ is $\frac{1}{(n-1)!}x^{n-1}$. So the leading term of $A_2(x,n)$ is
  \begin{equation*}
   \left(\sum_{\nu=0}^n\binom{n}{\nu}\right)^2\cdot \frac{x^{2(n-1)}}{(n-1)!^2}\:=\: \frac{2^{2n}}{(n-1)!^2}\cdot x^{2(n-1)}\:.
  \end{equation*}
Evaluating $A_2(x,n)$ at $x=0$ reduces the sum to the terms for $\nu=0$ and $n$, so $A_2(0,n)=(1+(-1)^{n-1})^2$. Similarly, at $x=1$ only the terms for $\nu=0$ and $1$ are non-zero, 
and we obtain  $A_2(1,n)=4n^2$. 
\end{proof}
We now prove Theorem~\ref{thm_summation_operator}.
\begin{proof}[Proof of Theorem~\ref{thm_summation_operator}]
We show (\ref{summation_equation_to_show}) for at integer places $x=n>m$. Then, as both  the image $S_1^2 P(m,x)$ and $A_1(m,x)$ are polynomials, they must be equal.
Let $n>m$ be an integer.
 By  changing the summation index we obtain
 \begin{align*}
  A_1(m,n)&=\left(\sum_{\nu=0}^m \binom{n-1+\nu}{n-1}\binom{n}{m-\nu}\right)\cdot\left(\sum_{\mu=0}^{m} \binom{n-1+m-\mu}{n-1}\binom{n}{\mu}\right)\:.
 \end{align*}
Let
\begin{equation}\label{formula_summand}
 a(\nu,\mu,m,n)\:=\:\binom{n-1+\nu}{\nu}\binom{n}{m-\nu} \binom{n-1+m-\mu}{m-\mu}\binom{n}{\mu}\:,
\end{equation}
so that $A_1(m,n)\:=\:\sum_{\nu,\mu=0}^m  a(\nu,\mu,m,n)$.
On the other hand, using the notation of Remark~\ref{remarks}~(iii),
\begin{align*}
 S_1^2 P(m,n) &= \sum_{\nu,\mu=0}^m D_n(\nu+1,n-1-\mu)\cdot D_n(m-\mu+1,n-1-m+\nu)\\
 &+2\cdot \sum_{\nu,\mu=0}^{m-1} D_n(\nu+1,n-1-\mu)\cdot D_n(m-1-\mu+1,n-1-(m-1)+\nu)\\
 &+ \sum_{\nu,\mu=0}^{m-2} D_n(\nu+1,n-1-\mu)\cdot D_n(m-2-\mu+1,n-1-(m-2)+\nu)\:.
\end{align*}
An index shift $\nu\mapsto \nu+1$, $\mu\mapsto \mu+1$ in the last sum and in one of the two second sums yields
\begin{align*}
 S_1^2 P(m,n) &= \sum_{\nu,\mu=0}^m D_n(\nu+1,n-1-\mu)\cdot D_n(m-\mu+1,n-1-m+\nu)\label{fette_summe}\\
 &+\sum_{\nu,\mu=0}^{m-1} D_n(\nu+1,n-1-\mu)\cdot D_n(m-\mu,n-m+\nu)\\
 &+\sum_{\nu,\mu=1}^{m} D_n(\nu,n-\mu)\cdot D_n(m-\mu+1,n-1-m+\nu)\\
 &+ \sum_{\nu,\mu=1}^{m-1} D_n(\nu,n-\mu)\cdot D_n(m-1-\mu+1,n-m+\nu)\:.
\end{align*}
Hence  equation~(\ref{summation_equation_to_show}) follows from Lemma~\ref{lemma_summation}.
The identity $A_1(m,n)=A_2(m,n)$ for integers $m,n>0$ is obvious by substituting $\nu\mapsto m-\nu$ in  $A_1(m,n)$ and then extending the sums in both $A_1(m,n)$ and $A_2(m,n)$ to $\max\{m,n\}$.
Notice that $A_2(m,n)=A(m,n)$ by definition of $A$ in the introduction, so formula (\ref{summation_equation_to_show_discrete}) is proved. 
Hence equation~(\ref{summation_equation_2_to_show}) holds for all integers $x=m>0$  and for all $n>0$ by the nearby symmetry of $P(m,n)$. 
Exchanging $n$ and $m$ we obtain (\ref{summation_equation_2_to_show_discrete}).
Therefore, both being polynomials of degree $2(n-1)$ in $x$,  the functions $A_2(x,n)$ and $S_1^2P(n,x)$ coincide.
\end{proof}
\begin{lem}\label{lemma_summation}
For $n>m$
  the numbers $a(\nu,\mu,m,n)$ defined in formula (\ref{formula_summand}) satisfy
 \begin{equation*}
  a(\nu,\mu,m,n)\:=\: b(\nu,\mu,m,n)\:,
 \end{equation*}
 where $b(\nu,\mu,m,n)$ denotes  the following sum
 \begin{align*}
  &D_n(\nu+1,n-1-\mu)\cdot D_n(m-\mu+1,n-1-m+\nu)\\
  &+(1-\delta_{\nu,m})\cdot(1-\delta_{\mu,m})\cdot D_n(\nu+1,n-1-\mu)\cdot D_n(m-\mu,n-m+\nu)\\
  &+(1-\delta_{\nu,0})\cdot(1-\delta_{\mu,0})\cdot D_n(\nu,n-\mu)\cdot D_n(m-\mu+1,n-1-m+\nu)\\
  &+(1-\delta_{\nu,0})\cdot(1-\delta_{\mu,0})(1-\delta_{\nu,m})\cdot(1-\delta_{\mu,m})\cdot D_n(\nu,n-\mu)\cdot D_n(m-\mu,n-m+\nu)\:.
 \end{align*}
\end{lem}
\begin{proof}[Proof of Lemma~\ref{lemma_summation}]
 Using the definition of $D_n(\alpha+1,\beta)$ (see~\ref{Weyl-dimension}), the lemma follows by straight forward calculations distinguishing the cases $\nu,\mu$ equal to $0$, $m$, or generic.
 We exemplify this in the generic case $0<\nu,\mu<m$. The sum $b(\nu,\mu,m,n)$ then is
 \begin{align*}
  &\frac{n^2}{(n+\nu-\mu)^2}\binom{n+\nu}{\nu}\binom{n-1}{n-1-\mu}\binom{n+m-\mu}{m-\mu}\binom{n-1}{n-1-m+\nu}\\
  &+\frac{n^2}{(n+\nu-\mu)^2}\binom{n+\nu}{\nu}\binom{n-1}{n-1-\mu}\binom{n+m-1-\mu}{m-1-\mu}\binom{n-1}{n-m+\nu}\\
  &+\frac{n^2}{(n+\nu-\mu)^2}\binom{n+\nu-1}{\nu-1}\binom{n-1}{n-\mu}\binom{n+m-\mu}{m-\mu}\binom{n-1}{n-1-m+\nu}\\
  &+\frac{n^2}{(n+\nu-\mu)^2}\binom{n+\nu-1}{\nu-1}\binom{n-1}{n-\mu}\binom{n+m-1-\mu}{m-1-\mu}\binom{n-1}{n-m+\nu}\:.
 \end{align*}
Summing the first and the second line as well as the third and forth we obtain
\begin{equation*}
 \frac{n^3}{(n+\nu-\mu)}\left[\binom{n+\nu}{\nu}\binom{n-1}{n-1-\mu}+\binom{n+\nu-1}{\nu-1}\binom{n-1}{n-\mu}\right]\:,
\end{equation*}
which is
\begin{equation*}
 \frac{n^2\cdot(n+m-\mu-1)!(n+\nu-1)!}{\nu!(m-\nu)!(n-m+\nu)!\mu!(m-\mu)!(n-\mu)!}\:.
\end{equation*}
Expanding by $(n-1)!^2$ we see that this fraction   is equal to $a(\nu,\mu,m,n)$.
\end{proof}
\section{Euler operator}\label{sec_Euler_operator}
In this section we study the summation operator $\widetilde S$ defined by
\begin{equation*}
 \widetilde S f(x)\:=\: f(x+1)+f(x)\:.
\end{equation*}
On polynomial rings the operators $S$ and $\widetilde S$ are related by $\widetilde S f(x)=S f(x+\frac{1}{2})$.
The Euler operator $\widetilde E$, that is the operator inverse to $\widetilde S$,  is given on polynomials by recursion, $\widetilde E(x^0)=\frac{1}{2}$, and for $n>0$
\begin{equation*}
 \widetilde E(x^n)\:=\: \frac{1}{2}x^n-\frac{1}{2}\sum_{j=0}^{n-1}\binom{n}{j}\widetilde E(x^j)\:.
\end{equation*}
Indeed
\begin{equation*}
 \widetilde S(\frac{1}{2}x^n)\:=\:\frac{1}{2}(x+1)^n+\frac{1}{2}x^n\:=\: x^n+\frac{1}{2}\cdot\sum_{j=0}^{n-1}\binom{n}{j}x^j\:.
\end{equation*}
We may also determine the polynomials $\widetilde E(x^n)$ by their generating series. Define polynomials $e_n(x)$ by
\begin{equation*}
 \sum_{n=0}^\infty e_n(x)\frac{t^n}{n!}\:=\:\frac{e^{xt}}{e^t+1}\:.
\end{equation*}
It holds $\widetilde S \bigl(\frac{e^{xt}}{e^t+1}\bigr)=e^{xt}$, hence the polynomials satisfy 
\begin{equation*}
 \widetilde S e_n(x)\:=\:x^n\:.
\end{equation*}
The polynomials $e_n(x)=\widetilde E(x^n)$ are the well-known Euler polynomials.
In particular, $e_1(x)=\frac{1}{2}x-\frac{1}{4}$, $e_2(x)=\frac{1}{2}x^2-\frac{1}{2}x$, $e_3(x)=\frac{1}{2}x^3-\frac{3}{4}x^2+\frac{1}{8}$, and $e_4(x)=\frac{1}{2}x^4-x^3+\frac{1}{2}x$.
The values $E_{2m}=2^{2m+1}e_{2m}(\frac{1}{2})$ are  the alternating Euler numbers, 
$E_0=1$, $E_2=-1$, $E_4=5$, $E_6=-61$, etc.
Whereas for all $m\geq 0$ it holds $e_{2m+1}(-\frac{1}{2})=0$.

Define the second order Euler polynomials $e_n^{[2]}(x)$ by the generating series 
\begin{equation*}
 \sum_{n=0}^\infty e^{[2]}_n(x)\frac{t^n}{n!}\:=\:\frac{e^{xt}}{(e^t+1)^2}\:.
\end{equation*}
It holds $(\widetilde S)^2 e^{[2]}_n(x)=\widetilde S e_n(x)=x^n$. The  polynomials $e^{[2]}_n(x)$ are determined from the first order ones $e^{[1]}_n=e_n(x)$
by Cauchy product expansion
\begin{equation*}
 e^{[2]}_n(x)\:=\:\sum_{k=0}^n \binom{n}{k}e_{n-k}(x)e_k(0)\:.
\end{equation*}

On the other hand, we may determine $\widetilde E(f)$ for a polynomial $f$ as follows.
Define a sequence $F_k$, $k\in\NN_0$, by $F_0=0$, and for all $k\geq 0$ 
\begin{equation*}
 F_{k+1}\:=\:f(k)-F_k\:.
\end{equation*}
It holds
\begin{equation*}
 F_{k+1}+F_k\:=\: f(k)\:,
\end{equation*}
so the sequence $F_k$ is a  solution of the sequence of discrete equations $\widetilde SF_k=f(k)$. Any other solution differs from $F_k$ only by a sequence $(-1)^k\cdot c$ for some constant $c$.
Let $G(x)$ be the interpolation polynomial of degree at most $n=\deg f$ of the values $G(k)=F_k$ for $k=0,\dots,n$. It is given by Lagrange interpolation
\begin{equation*}
 G(x)\:=\:\sum_{j=0}^n F_j\cdot\prod_{k=0,k\not= j}^n\frac{x-k}{j-k}\:=\:\sum_{j=0}^n \frac{(-1)^{n-j}F_j}{j!(n-j)!}\prod_{k=0,k\not= j}^n(x-k)\:
\end{equation*}
and satisfies the summation equation
\begin{equation*}
 G(k+1)+G(k)\:=\: f(k)
\end{equation*}
for $k=0,1,\dots n-1$.
In order to obtain the polynomial solution $F(x)=\widetilde E(f(x))$  we must add a multiple of the polynomial $B_n(x)$ of degree $n$ interpolating the values $B_n(k)=(-1)^k$ for $k=0,1,\dots,n$,
\begin{equation}\label{general_solution_F}
 F(x)\:=\: G(x)+c\cdot B_n(x)\:,
\end{equation}
where the constant $c$ is determined by the summation equation at $x=n$
\begin{equation*}
 F(n+1)+F(n)=G(n+1)+G(n)+c\cdot(B_n(n+1)+B_n(n))\:=\: f(n)\:.
\end{equation*}
Because the polynomial $B_n(x)$ is given explicitly in Lagrange form
\begin{equation*}
 B_n(x)\:=\: (-1)^{n}\sum_{j=0}^n \frac{1}{j!(n-j)!}\prod_{k=0,k\not= j}^n(x-k)\:,
\end{equation*}
it follows
\begin{equation*}
 B_n(n+1)+B_n(n)\:=\: (-1)^n\bigl( (2^{n+1}-1)+1\bigr)\:=\: (-1)^n2^{n+1}\:. 
\end{equation*}
We obtain
\begin{align*}
  (-1)^n2^{n+1}\cdot c&= f(n)-G(n)-G(n+1)\\
  &=\sum_{j=0}^{n}(-1)^{n-j}\left(f(j)-\binom{n+1}{j}F_j\right)\\
  &=(-1)^n\sum_{j=0}^n\Bigl((-1)^jf(j)+\binom{n+1}{j}\sum_{k=0}^{j-1}(-1)^kf(k)\Bigr)\:,
\end{align*}
or equivalently
\begin{equation}\label{gleichung_fuer_c}
 c\:=\:\frac{1}{2^{n+1}}\sum_{k=0}^n(-1)^kf(k)\sum_{j=k+1}^{n+1}\binom{n+1}{j}\:.
\end{equation}
We summarize.
\begin{prop}\label{urbild-polynom}
The preimage $F=\widetilde E(f)$ of the polynomial $f$ of degree $n$ under the summation operator $\widetilde S$ is the polynomial
\begin{equation*}
 F(x)\:=\:\sum_{j=0}^n \frac{\bigl(c+(-1)^{n-j}F_j\bigr)}{j!(n-j)!}\prod_{k=0,k\not= j}^n(x-k)\:,
\end{equation*}
where the constant $c=(-1)^{n}F(0)$ is given by (\ref{gleichung_fuer_c}), and the coefficients $F_j$ are determined by the recursion $F_0=0$, and $F_{k+1}=f(k)-F_k$ for $k=0,\dots,n$.
\end{prop}
Writing
\begin{equation*}
 f(x)\:=\: a_nx^n+\dots+a_1x+a_0\:,
\end{equation*}
for the solution polynomial  it follows 
\begin{equation*}
 F(x)\:=\:\frac{a_n}{2}x^n+\dots\:.
\end{equation*}
Comparing this with the highest coefficient of (\ref{general_solution_F}) we obtain
\begin{equation*}
 \frac{a_n\cdot n!}{2}\:=\:(-1)^n2^n\cdot c+(-1)^n\sum_{j=0}^n\binom{n}{j}(-1)^jF_j\:,
\end{equation*}
which by expanding $F_j=\sum_{k=0}^{j-1}(-1)^{j-1-k}f(k)$ is equivalent to a second formula for the constant $c$
\begin{equation}\label{2te_gleichung_fuer_c}
 c\:=\: \frac{(-1)^n n!}{2^{n+1}}\cdot a_n +\frac{1}{2^n}\sum_{k=0}^n(-1)^kf(k)\sum_{j=k+1}^{n}\binom{n}{j}\:.
\end{equation}
Simplifying the identity
 (\ref{gleichung_fuer_c})~$=$~(\ref{2te_gleichung_fuer_c}) yields the well-known expression for the leading coefficient $a_n$ of the polynomial $f$
\begin{equation}\label{formula_highest_coeff}
 (-1)^n n! \cdot a_n\:=\: \sum_{k=0}^n\binom{n}{k}(-1)^kf(k)\:.
\end{equation}
By this and  $(-1)^nc=F(0)$ being the constant coefficient of $F$, a number of non-obvious combinatorial identities arise. 
\begin{eg}
 \begin{itemize}
  \item [(a)]
  Let $f(x)=x^n$. We obtain $F_j=\sum_{l=0}^{j-1}(-1)^{j-1-l}l^n$ as well as
\begin{equation*}
 c\:=\:\frac{1}{2^{n+1}}\sum_{l=0}^n(-1)^ll^n\sum_{i=l+1}^{n+1}\binom{n+1}{i}\:.
\end{equation*}
For the coefficients $F_j+(-1)^jc$ of the solution polynomial
\begin{equation*}
 F(x)\:=\:\sum_{j=0}^n\bigl(F_j+(-1)^jc\bigr)\prod_{k\not= j}\frac{x-k}{j-k}
\end{equation*}
we obtain
\begin{align*}
 F_j+(-1)^jc&=(-1)^j\Bigg(\sum_{l=0}^n(-1)^ll^n\frac{\sum_{i=l+1}^{n+1}\binom{n+1}{i}}{2^{n+1}}-\sum_{l=0}^{j-1}(-1)^ll^n\Bigg)\\
 &=\frac{(-1)^j}{2^{n+1}}\Bigg(\sum_{l=j}^n(-1)^ll^n\sum_{i=l+1}^{n+1}\binom{n+1}{i}-\sum_{l=0}^{j-1}(-1)^ll^n\sum_{i=0}^l\binom{n+1}{i}\Bigg)\:.
\end{align*}
Computing  the leading coefficient of $F(x)=\frac{1}{2}x^n+\dots$ from the above formula  we obtain the following identity.
 For all integers $n>0$ it holds true
 \begin{equation*}
  (-1)^n2^nn!\:=\:\sum_{j=0}^n\binom{n}{j}\Bigg(\sum_{l=j}^n(-1)^ll^n\sum_{i=l+1}^{n+1}\binom{n+1}{i}-\sum_{l=0}^{j-1}(-1)^ll^n\sum_{i=0}^l\binom{n+1}{i}\Bigg)\:.
 \end{equation*}
\item [(b)]
 Let
 \begin{equation*}
  f(x)\:=\:\begin{bmatrix}
            x\\n-1
           \end{bmatrix}
\:=\:\frac{1}{(n-1)!}\cdot x(x-1)\cdots(x-(n-1)+1)\:.
 \end{equation*}
Then the polynomial solution $F$ of the summation equation $\widetilde S F(x)=f(x)$ is 
\begin{equation*}
 F(x)\:=\:\frac{1}{2^n(n-1)!}\sum_{j=0}^{n-1}\binom{n-1}{j}\prod_{k=0,k\not= j}^{n-1}(x-k)\:.
\end{equation*}
The values $F_k$, $k=0,\dots,n-1$ in the consideration above are zero. Hence the associated interpolation polynomial $G$ of degree at most $n-1$ of these values is zero, too.
 It follows $F(x)=c\cdot B_{n-1}(x)$, where the constant $c$ given by (\ref{gleichung_fuer_c}) is 
 \begin{equation*}
  c\:=\: \frac{1}{2^{n}}(-1)^{n-1}f(n-1)\binom{n}{n}\:=\: \frac{(-1)^{n-1}}{2^n}\:.
 \end{equation*}
 \end{itemize}
\end{eg}
We apply the above method once more. We will use the notion of super Catalan numbers given by Gessel~\cite{gessel}.
\begin{defn}
 For integers $m,k\geq 0$ define the super Catalan number
 \begin{equation*}
 C(m,k)\:=\:\frac{(2m)!\cdot (2k)!}{2\cdot m!\cdot k!\cdot (m+k)!}\:.
\end{equation*}
\end{defn}
Super Catalan numbers are integers and satisfy the  summation equation \cite[p.~191]{gessel}
\begin{equation}\label{super_catalan}
 C(m+1,k)+C(m,k+1)\:=\: 4\cdot C(m,k)\:.
\end{equation}
\begin{prop}\label{prop-preimages}
(a) 
For $\nu=0,1,\dots,n$ let
\begin{equation*}
 f(x,\nu,n)\:=\:\begin{bmatrix}
            x\\n
           \end{bmatrix}\begin{bmatrix}
            x-\nu\\n
           \end{bmatrix}\:.
\end{equation*}
Then  the polynomial solution of the equation $\widetilde S F(x,\nu,n) =f(x,\nu,n)$ is
\begin{equation*}
 F(x,\nu,n)\:=\:\sum_{j=0}^{2n}\bigl(F_{j}(\nu,n)+(-1)^jc(\nu,n)\bigr)\prod_{k=0,k\not= j}^{2n}\frac{x-k}{j-k}\:,
\end{equation*}
where 
\begin{equation*}
 F_{j}(\nu,n)\:=\:(-1)^{j-1}\sum_{k=0}^{j-1}(-1)^{k}\binom{k}{n}\binom{k-\nu}{n}\:.
\end{equation*}
Here the constants $c(\nu,n)$ 
satisfy the following two    recursion formulas for $2\leq\nu\leq n$
\begin{equation}\label{recursion_for_c_nu_one}
 c(\nu-2,n)\:=\: c(\nu,n)-c(\nu-1,n-1)\:,
\end{equation}
and
\begin{equation}\label{recursion_for_c_nu}
  c(\nu-2,n)\:=\: -c(\nu,n)+\frac{\nu}{n}c(\nu-1,n-1)\:.
\end{equation}
For  $0\leq\nu\leq n$ they are explicitly given  by
\begin{equation}\label{c_nu_explicitly}
 c(\nu,n)\:=\:\left\{\begin{array}{ll}
                        0\:,&\textrm{ if } \nu=n-1-2\mu\\
                        \frac{(-1)^\mu}{ 2^{2n}}\cdot C(n-\mu,\mu)\:,&\textrm{ if } \nu=n-2\mu
                       \end{array}\right.\:.
\end{equation}
(b)
For
 \begin{equation*}
 f(x,n+1,n)\:=\:\begin{bmatrix}
            x\\n
           \end{bmatrix}\begin{bmatrix}
            x-n-1\\n
           \end{bmatrix}\:
\end{equation*}
the polynomial solution of the equation $\widetilde S F(x,n+1,n) =f(x,n+1,n)$ is
\begin{equation*}
 F(x,\nu,n)\:=\:\frac{1}{2}\sum_{j=0}^{n}(-1)^j\prod_{k=0,k\not= j}^{2n}\frac{x-k}{j-k}+\frac{1}{2}\sum_{j=n+1}^{2n}(-1)^{j+1}\prod_{k=0,k\not= j}^{2n}\frac{x-k}{j-k}\:.
\end{equation*} 
\end{prop}
For example we  obtain
\begin{equation*}
 F(x,n,n)\:=\: \frac{C(n,0)}{2^{2n}}\cdot B_{2n}(x)\:,
\end{equation*}
while
\begin{equation*}
 F(x,n-1,n)\:=\: \binom{2n-1}{n}\cdot \begin{bmatrix}
                                               x\\2n
                                              \end{bmatrix}\:.
\end{equation*}
We emphasize that the values of $F(x,\nu,n)$ at $x=k$ for $0\leq k\leq 2n$ are explicitly given by
\begin{equation*}
 F(k,\nu,n)\:=\:(-1)^kc(\nu,n)+\sum_{j=n+\nu}^{k-1}(-1)^{k-1+j}\binom{j}{n}\binom{j-\nu}{n}\:.
\end{equation*}
In particular, for $\nu=n-1-2\mu$ the polynomial $F(x,\nu,n)$ has zeros in $x=0,1,\dots,n-\nu$.

\begin{cor}\label{cor-on-super-Catalan}
The super Catalan numbers $C(n,\mu)$ satisfy the following identities.
 \begin{itemize}
  \item [(a)]
  For all $0\leq \mu\leq n$ 
  \begin{equation*}
   2\cdot C(n,0)\:=\:\sum_{k=0}^{2n}\binom{2n}{k}(-1)^k\binom{k}{n}\binom{k-\mu}{n}\:.
  \end{equation*}
\item[(b)]
For all $0\leq\mu\leq \lfloor\frac{n-1}{2}\rfloor$ 
\begin{equation*}
 C(n,0)\:=\:\sum_{k=0}^{2n}(-1)^{k+1}\binom{k}{n}\binom{k+1+2\mu-n}{n}\sum_{j=k+1}^{2n}\binom{2n}{j}\:,
\end{equation*}
whereas for all $0\leq \mu\leq \lfloor\frac{n}{2}\rfloor$
\begin{equation*}
 (-1)^\mu\cdot C(n-\mu,\mu)-C(n,0)\:=\:\sum_{k=0}^{2n}(-1)^{k}\binom{k}{n}\binom{k+2\mu-n}{n}\sum_{j=k+1}^{2n}\binom{2n}{j}\:.
\end{equation*}
 \end{itemize}
 Notice that in each sum, the summands actually are zero for $0\leq k\leq n+\mu$.
\end{cor}
\begin{proof}
 Part (a) is given by identity (\ref{formula_highest_coeff}) for the leading coefficients of the polynomials $f(x,\nu,n)$.
 Part (b) is given by formula (\ref{2te_gleichung_fuer_c}) for the constants $c(\nu,n)$ of the polynomials $f(x,\nu,n)$ which are also determined by Proposition~\ref{prop-preimages}.
 Part (b) follows in particular from equation (\ref{equation_usful-for-cor}) below by inserting the special values of $c(\nu-2,n)$ given in Proposition~\ref{prop-preimages}.
\end{proof}

\begin{proof}[Proof of Proposition~\ref{prop-preimages}]
Part (a): Notice that
\begin{equation*}
 f(x,\nu,n)\:=\:\frac{1}{n!^2}\prod_{j=0}^{\nu-1}(x-j)(x-n-j)\cdot\prod_{k=1}^{n-\nu}(x-(\nu-1)-k)^2\:.
\end{equation*}
Hence, in Proposition~\ref{urbild-polynom}, the  series $F_j=F_j(\nu,n)$ is given by 
\begin{equation*}
 F_j(\nu,n)\:=\:(-1)^{j-1}\sum_{k=0}^{j-1}(-1)^{k}\binom{k}{n}\binom{k-\nu}{n}\:,
\end{equation*}
where the summands vanish for $k<n+\nu$. In particular, $F_j(\nu,n)=0$ for $0\leq j\leq n+\nu$.
Let
\begin{equation*}
 G(x,\nu,n)\:=\:\sum_{j=0}^{2n}F_j(\nu,n)\prod_{k\not= j}\frac{x-k}{j-k}\:=\: \sum_{j=n+1+\nu}^{2n}F_j(\nu,n)\prod_{k\not= j}\frac{x-k}{j-k}\\
\end{equation*}
be the associated interpolation polynomial of degree $2n$.
We obtain the polynomial solution 
\begin{equation*}
 F(x,\nu,n)\:=\:\sum_{j=0}^{2n}\bigl(F_j(\nu,n)+(-1)^jc(\nu,n)\bigr)\prod_{k\not= j}\frac{x-k}{j-k}\:,
\end{equation*}
where the
constant $c(\nu,n)$ is determined 
by (\ref{2te_gleichung_fuer_c})
 \begin{equation}\label{rattenschwanz}
  \frac{(2n)!}{2^{2n+1}n!^2}+\frac{1}{2^{2n}}\sum_{k=n+\nu}^{2n}(-1)^k\binom{k}{n}\binom{k-\nu}{n}\sum_{j=k+1}^{2n}\binom{2n}{j}\:.
 \end{equation}
Evaluating this formula for $\nu=n$ we obtain
\begin{equation*}
 c(n,n)\:=\:\frac{(2n)!}{2^{2n+1}n!^2}\:=\:\frac{1}{2^{2n}}\cdot C(n,0)\:.
\end{equation*}
Evaluation at $\nu=n-1$ yields
\begin{equation*}
 c(n-1,n)\:=\:\frac{(2n)!}{2^{2n+1}n!^2}-\frac{1}{2^{2n}}\binom{2n-1}{n}\:=\:0\:.
\end{equation*}
We  use these special values  to prove (\ref{c_nu_explicitly}) by increasing induction on $n$ and decreasing induction on $\nu$ using recursion formula (\ref{recursion_for_c_nu}).
But first observe that recursion formula (\ref{recursion_for_c_nu_one}) follows from  recursion formula  (\ref{super_catalan}) for super Catalan numbers once the explicit 
values (\ref{c_nu_explicitly}) hold true.
For the above induction  we have to prove (\ref{recursion_for_c_nu}).
Observe that
 \begin{equation*}
  \binom{k+1}{n}\binom{k+1-\nu}{n-1}-\binom{k}{n-1}\binom{k+1-\nu}{n-1}\:=\: \frac{n+\nu}{n} \binom{k}{n-1}\binom{k+1-\nu}{n-1}\:.
 \end{equation*}
Accordingly,
\begin{align}
 \binom{k}{n}\binom{k-\nu+2}{n}&=\left[\binom{k+1}{n}-\binom{k}{n-1}\right]\cdot\left[\binom{k+1-\nu}{n}+\binom{k-(\nu-1)}{n-1}\right]\nonumber\\
 &=\binom{k+1}{n}\binom{k+1-\nu}{n}+\frac{\nu}{n}\binom{k}{n-1}\binom{k-(\nu-1)}{n-1}\:.\label{binom_coeff_identity}
\end{align}
Using the definition of the super Catalan number, by (\ref{rattenschwanz})  we know
\begin{equation}\label{equation_usful-for-cor}
 2^{2n}c(\nu-2,n)\:=\: C(n,0)+\!\!\!\!\!\sum_{k=n+\nu-2}^{2n}(-1)^k\binom{k}{n}\binom{k-\nu+2}{n}\sum_{j=k+1}^{2n}\binom{2n}{j}\:.
\end{equation}
We split this expression according to the identity for binomial coefficients (\ref{binom_coeff_identity}).
For the first part we obtain
\begin{align}
 &\quad  C(n,0)+\sum_{k=n+\nu-2}^{2n}(-1)^k\binom{k+1}{n}\binom{k+1-\nu}{n}\sum_{j=k+1}^{2n}\binom{2n}{j}\nonumber\\
 =&\quad C(n,0)-\sum_{k=n+\nu}^{2n+1}(-1)^k\binom{k}{n}\binom{k-\nu}{n}\sum_{j=k}^{2n}\binom{2n}{j}\nonumber\\
 =&\quad -C(n,0)-\sum_{k=n+\nu}^{2n}(-1)^k\binom{k}{n}\binom{k-\nu}{n}\sum_{j=k+1}^{2n}\binom{2n}{j}\nonumber\\
 =&\quad(-1)\cdot 2^{2n}\cdot c(\nu,n)\:,\nonumber
\end{align}
where we used  formula (\ref{formula_highest_coeff}) for the highest coefficient $(n!)^{-2}$ of the polynomial $f(x,\nu,n)$.
For the second part we obtain 
\begin{align}
 &\quad \sum_{k=n+\nu-2}^{2n}(-1)^k\binom{k}{n-1}\binom{k-(\nu-1)}{n-1}\sum_{j=k+1}^{2n}\binom{2n}{j}\nonumber\\
 =&\quad 2\cdot\sum_{k=n-1+(\nu-1)}^{2(n-1)}(-1)^k\binom{k}{n-1}\binom{k-(\nu-1)}{n-1}\sum_{j=k+1}^{2n-1}\binom{2n-1}{j}\nonumber\\
 &\quad +\sum_{k=n-1+(\nu-1)}^{2n}\binom{2n-1}{k}(-1)^k\binom{k}{n-1}\binom{k-(\nu-1)}{n-1}\nonumber\\
 =&\quad 4\cdot 2^{2(n-1)}\cdot c(\nu-1,n-1)+2C(n-1,0)\nonumber\\
 &\quad +\sum_{k=n-1+(\nu-1)}^{2n-1}\binom{2(n-1)}{k-1}(-1)^k\binom{k}{n-1}\binom{k-(\nu-1)}{n-1}\nonumber\\
 =&\quad  4\cdot 2^{2(n-1)}\cdot c(\nu-1,n-1)+2\cdot C(n-1,0)-2\cdot C(n-1,0)\nonumber\\
 =&\quad  2^{2n}\cdot c(\nu-1,n-1)\:,\nonumber
\end{align}
where we  used that the sum
\begin{align*}
 &\quad \sum_{k=n-1+(\nu-1)}^{2n-1}\binom{2(n-1)}{k}(-1)^k\binom{k}{n-1}\binom{k-(\nu-1)}{n-1}\\
 =&\quad \sum_{k=n-1+(\nu-1)-1}^{2(n-1)}\binom{2(n-1)}{k}(-1)^{k+1}\binom{k+1}{n-1}\binom{k+1-(\nu-1)}{n-1}\\
 =&\quad -\sum_{k=0}^{2(n-1)}\binom{2(n-1)}{k}(-1)^{k}g(k)\\
 =&\quad -2\cdot C(n-1,0)
\end{align*}
equals up to the constant $(2(n-1))!$ the leading coefficient of the polynomial $g(x)=f(x+1,\nu-1,n-1)$.
So putting the two parts together  keeping in mind (\ref{binom_coeff_identity}), we obtain
\begin{equation*}
 2^{2n}c(\nu-2,n)\:=\:(-1)\cdot 2^{2n}c(\nu,n)+\frac{ 2^{2n}\cdot\nu}{n}\cdot c(\nu-1,n-1)\:.
\end{equation*}
We have proved recursion formula (\ref{recursion_for_c_nu}).
In order to finish the induction argument, by hypothesis we assume that (\ref{c_nu_explicitly}) holds true for $c(\nu,n)$ as well as for $c(\nu-1,n-1)$.
If $\nu$ is of the form $\nu=n-1-2\mu$ these two constants are zero, so (\ref{recursion_for_c_nu}) implies $c(\nu-2,n)=0$.
If $\nu=n-2\mu$ we obtain for the right hand side of  (\ref{recursion_for_c_nu})
\begin{align*}
&  2^{-2n}(-1)^{\mu+1}\Bigl(C(n-\mu,\mu)-\frac{n-2\mu}{n-1}\cdot4\cdot C(n-1,\mu,\mu)\Bigr)\\
& =\frac{(-1)^{\mu+1}(2(n-1-\mu))!(2\mu)!}{2^{2n}(n-1-\mu)!\mu!n!}\bigl(2(2n+1-2\mu)-4(n-2\mu)\bigr)\\
& =\frac{(-1)^{\mu+1}(2(n-1-\mu))!(2\mu)!2(2\mu+1)(\mu+1)}{2^{-2n}(n-1-\mu)!(\mu+1)!n!}\\
& =\frac{(-1)^{\mu+1}}{2^{2n}}C(n-(\mu+1),\mu+1)\:,
\end{align*}
which must equal the left hand side $c(n-2(\mu+1),n)$ of (\ref{recursion_for_c_nu}).

Part (b):
 We proceed again by Proposition~\ref{urbild-polynom} to obtain  the values
 \begin{equation*}
  F_j(n+1,n)\:=\:\sum_{k=0}^{j-1}(-1)^{j-1-k}f(k,n+1,n)\:=\:\left\{\begin{array}{ll}
                                                                    0&\textrm{ if } j=0,\dots,n\\
                                                                    (-1)^{j-1}&\textrm{ if } j=n+1,\dots,2n
                                                                   \end{array}\right.,
 \end{equation*}
as well as the constant
\begin{align*}
 c&=2^{-2n}\Bigl(\frac{(2n)!}{2n!^2}+\sum_{k=0}^{2n}(-1)^kf(k,n+1,n)\sum_{j=k+1}^{2n}\binom{2n}{j}\Bigr)\\
 &=2^{-2n}\Bigl(\frac{1}{2}\binom{2n}{n}+\sum_{j=n+1}^{2n}\binom{2n}{j}\Bigr)\:=\: \frac{1}{2}\:.
\end{align*}
This leads to the solution polynomial
\begin{align*}
 F(x,n+1,n)&=\sum_{j=0}^{2n}\bigl(F_j+(-1)^jc\bigr)\cdot\prod_{k=0,k\not=j}^{2n}\frac{x-k}{j-k}\\
 &=\frac{1}{2}\sum_{j=0}^{n}(-1)^j\prod_{k=0,k\not= j}^{2n}\frac{x-k}{j-k}+\frac{1}{2}\sum_{j=n+1}^{2n}(-1)^{j+1}\prod_{k=0,k\not= j}^{2n}\frac{x-k}{j-k}\:.
\end{align*}
\end{proof}

\section{Mixed summation operator}\label{sec_mixed}
We study the action of the mixed summation operator $\widetilde S_1\widetilde S_2$ in two variables on the functions $P(n,m)$.
\begin{prop}\label{definierende_eigenschaften_mixed}
For $m=1,2,3,\dots$ there is a unique family of poly\-no\-mials~$Q(m,x)$  in $\QQ[x]$ satisfying the following properties.
 \begin{itemize}
  \item [(i)] For all $m$ the degree of the polynomial is $\deg_x Q(m,x)= 2(m-1)$.
  \item [(ii)]
  For all $m$ the leading coefficient of $Q(m,x)$ is $\frac{2^{2m-1}}{(m-1)!^2}$.
  \item [(iii)]
   $Q(m,-x)=Q(m,x+1)$ holds for all $m\in\NN$.
  \item[(iv)] The function  $Q(m,n)$ is  symmetric  on $\NN\times\NN$, i.e. $Q(m,n)=Q(n,m)$.
 \end{itemize}
\end{prop}
\begin{proof}
Similarly to Proposition~\ref{definierende_eigenschaften} properties (i)--(iv) uniquely determine the family of polynomials by recursion. 
By properties (i) and (ii), $Q(1,x)=2$ is constant.
By property (iii), $Q(m,x)$ is an even polynomial in $x-\frac{1}{2}$, so by (ii)
\begin{equation*}
 Q(m,x)\:=\: \frac{2^{2m-1}}{(m-1)!^2}(x-\frac{1}{2})^{2(m-1)}+\sum_{j=0}^{m-2}a_{j,m}(x-\frac{1}{2})^{2j}\:.
\end{equation*}
Assuming  by recursion that the polynomials $Q(k,x)$ are defined for all $k=1,\dots,m-1$, the values $Q(m,k)$ are determined by (iv) for $k=1,\dots,m-1$.
This fixes $m-1$ values of the polynomial $\sum_{j=0}^{m-2}a_{j,m}y^j$ in $y=(x-\frac{1}{2})^{2}$ of degree at most $m-2$. Hence the coefficients $a_{j,m}$, $j=0,\dots,m-2$ are uniquely determined, 
and so is $Q(m,x)$. For example, $Q(2,x)=8(x-\frac{1}{2})^{2}$.
\end{proof}
In particular we obtain
\begin{small}
 \begin{align*}
  Q(1,x)&=2\:,\\
  Q(2,x)&=8(x-\frac{1}{2})^2\:,\\
  Q(3,x)&=8\bigl((x-\frac{1}{2})^2 + \frac{1}{4}\bigr)^2\:,\\
  Q(4,x)&=\frac{32}{9}(x-\frac{1}{2})^2\bigl((x-\frac{1}{2})^2+\frac{5}{4}\bigr)^2\:,\\
  Q(5,x)&=\frac{8}{9}\bigl((x-\frac{1}{2})^4 + \frac{7}{2}(x-\frac{1}{2})^2 + \frac{9}{16}\bigr)^2\:,\\
  Q(6,x)&=\frac{32}{225}(x-\frac{1}{2})^2\bigl((x-\frac{1}{2})^4 + \frac{15}{2}(x-\frac{1}{2})^2 + \frac{89}{16}\bigr)^2\:,\\
  Q(7,x)&=\frac{32}{2025}\bigl((x-\frac{1}{2})^2 + \frac{9}{4}\bigr)^2\bigl((x-\frac{1}{2})^4 + \frac{23}{2}(x-\frac{1}{2})^2 + \frac{25}{16}\bigr)^2\:,\\
  Q(8,x)&=\frac{128}{99225}(x-\frac{1}{2})^2\bigl((x-\frac{1}{2})^6 + \frac{91}{4}(x-\frac{1}{2})^4 + \frac{1519}{16}(x-\frac{1}{2})^2 + \frac{3429}{64}\bigr)^2\:.
 \end{align*}
\end{small}

\begin{prop}\label{properties_of_Q}
 The family of polynomials $Q(m,x)$ of Proposition~\ref{definierende_eigenschaften_mixed} is given by
 \begin{equation*}
   Q(m,x)\:=\:P(m,x)+P(m,x-1)+P(m-1,x)+P(m-1,x-1)\:,
 \end{equation*}
where $P(m,x)$ are the polynomials defined in Proposition~\ref{the_polynomials}. 
The polynomials have the following properties.
\begin{itemize}
 \item [(a)]
  The constant coefficient $Q(m,0)=2$ equals two for all $m\geq 1$.
  \item [(b)] The values $Q(m,n)$ define a symmetric function $Q:\NN\times\NN\to\ZZ$.
  \item [(c)]  For all $m\geq 1$ the polynomial $Q(m,x)$ is even in $x-\frac{1}{2}$
  \begin{equation*}
   Q(m,x)\:=\: \frac{2^{2m-1}}{(m-1)!^2}(x-\frac{1}{2})^{2(m-1)}+\dots+a_{1,m}(x-\frac{1}{2})^{2}+a_{0,m}\:.
  \end{equation*}
\item [(d)] For the Euler operator $\widetilde E$ and the polynomials $A_2(x,m)$ defined in Theorem~\ref{thm_summation_operator} we obtain
\begin{equation*}
 Q(m,x)\:=\: \widetilde E\Bigl(A_2(x,m)+A_2(x,m-1)\Bigr)\:.
\end{equation*}
\end{itemize}
\end{prop}
\begin{table}
\caption{Initial values of the function $Q:\NN\times\NN\to\ZZ$.}
\begin{small}
\begin{tabular}{c|c|c|c|c|c|c|c|c}
 $Q(m,n)$&$1$&$2$&$3$&$4$&$5$&$6$&$7$&$8$\\
 \hline
$1$&$2$&$2$&$2$&$2$&$2$&$2$&$2$&$2$\\
\hline
$2$&$2$&$18$&$50$&$98$&$162$&$242$&$338$&$450$\\
\hline
$3$&$2$&$50$&$338$&$1250$&$3362$&$7442$&$14450$&$25538$\\
\hline
$4$&$2$&$98$&$1250$&$7938$&$33282$&$106722$&$284258$&$661250$\\
\hline
$5$&$2$&$162$&$3362$&$33282$&$206082$&$927522$&$3323042$&$10044162$\\
\hline
$6$&$2$&$242$&$7442$&$106722$&$927522$&$5664978$&$26688818$&$103190978$\\
\hline
$7$&$2$&$338$&$14450$&$284258$&$3323042$&$26688818$&$161604242$&$786061250$\\
\hline
$8$&$2$&$450$&$25538$&$661250$&$10044162$&$103190978$&$786061250$&$4731504642$\\
\end{tabular}
\end{small}
\end{table}
In Proposition~\ref{prop_x=1/2_values_polynomials}~(c) we determine the constant coefficient $a_{0,m}$ from Proposition~\ref{properties_of_Q}~(c).
By numerical evidence for $1\leq m\leq 100$, we suppose that the polynomial $2\cdot Q(m,x)$  is a square for all $m$.
\begin{proof}[Proof of Proposition~\ref{properties_of_Q}]
 The polynomials $P(m,x)$ satisfy the properties of Proposition~\ref{definierende_eigenschaften}, and by Remark~\ref{remarks}(i) it holds $\deg_xP(m,x)=2(m-1)$ for all $m\geq 1$.
Hence  the sums $Q(m,x)=P(m,x)+P(m,x-1)+P(m-1,x)+P(m-1,x-1)$ satisfy properties (i),(iii), and (iv) of Proposition~\ref{definierende_eigenschaften_mixed}.
From the definition of $P(m,x)$ for $m>0$ we determine the leading coefficient of $P(m,x)$
\begin{equation*}
 \sum_{\nu,\mu=0}^{m-1}\frac{1}{\nu!(m-1-\nu)!\mu!(m-1-\mu)!}\:=\:\left(\frac{2^{m-1}}{(m-1)!}\right)^2\:.
\end{equation*}
It follows that the leading coefficient of $Q(m,x)$ is twice this number, hence equals $\frac{2^{2m-1}}{(m-1)!^2}$.
 Property (a) follows from Proposition~\ref{definierende_eigenschaften_mixed}, $Q(m,0)=Q(m,1)=Q(1,m)=2$.
 Property (b) follows because all the values $P(m,n)$ are integers, see Corollary~\ref{corrolar_1}.
 The shape of the polynomials $Q(m,x)$ given in (c) is provided again by  Proposition~\ref{definierende_eigenschaften_mixed}.
 Recall from Theorem~\ref{thm_summation_operator}
\begin{equation*}
 S_2^2 P(n,x)\:=\: P(n,x+1)+2P(n,x)+P(n,x-1)\:=\: A_2(x,n)\:.
\end{equation*}
For the shifted summation operator $\widetilde S P(n,x)=P(n,x+1)+P(n,x)$ we therefore obtain
\begin{equation*}
 \widetilde S\Bigl( P(n,x)+P(n,x-1)\Bigr)\:=\: A_2(x,n)\:.
\end{equation*}
Hence the polynomials $Q(m,x)$ in question are also given by the Euler operator $\widetilde E=\widetilde S^{-1}$
\begin{equation*}
 Q(m,x)\:=\: \widetilde E\Bigl(A_2(x,m)+A_2(x,m-1)\Bigr)\:.\qedhere
\end{equation*}
\end{proof}

Using of the summation operator
$Sf(x)=f(x+\frac{1}{2})+f(x-\frac{1}{2})$, the explicit formula for the polynomials $Q(m,x)$  given in Proposition~\ref{properties_of_Q} can be reformulated
\begin{equation}\label{Q-S-equation}
 Q(m,x+\frac{1}{2})\:=\: S\bigl(P(m,x)+P(m-1,x)\bigr)\:.
\end{equation}
 This suggests that all the families of polynomials we have defined in this paper should have interesting properties at half-integral places.
 We illustrate this by determining their values in $x=\frac{1}{2}$. We need the following lemma.
\begin{lem}\label{lemma_a1(m,1/2)}
For $m=1,2,3,\dots$ define the polynomials
\begin{equation*}
 a_1(m,x)\:=\: \sum_{\nu=0}^m\begin{bmatrix}x-1+\nu\\\nu\end{bmatrix}\begin{bmatrix}x\\m-\nu\end{bmatrix}\:.
\end{equation*}
Then the values in $x=\frac{1}{2}$ are given by 
\begin{equation*}
 a_1(2n,\frac{1}{2})\:=\:a_1(2n+1,\frac{1}{2})\:=\:\begin{bmatrix} n-\frac{1}{2}\\n\end{bmatrix}
\end{equation*}
for all $n\in\NN_0$.
\end{lem}
\begin{proof}[Proof of Lemma~\ref{lemma_a1(m,1/2)}]
We first notice that
 \begin{equation*}
  \begin{bmatrix} n-\frac{1}{2}\\n\end{bmatrix}\:=\: \frac{1}{2^{2n}}\binom{2n}{n}\:=\: (-1)^n\begin{bmatrix} -\frac{1}{2}\\n\end{bmatrix}\:.
 \end{equation*}
So the generating series of $\begin{bmatrix} n-\frac{1}{2}\\n\end{bmatrix}$ is given by the Taylor series
\begin{equation*}
 \sum_{n=0}^\infty x^n\begin{bmatrix} n-\frac{1}{2}\\n\end{bmatrix}\:=\:\sum_{n=0}^\infty (-x)^n\begin{bmatrix} -\frac{1}{2}\\n\end{bmatrix}\:=\:\bigl(1-x\bigr)^{-\frac{1}{2}}\:,
\end{equation*}
while 
\begin{equation*}
 \bigl(1+x\bigr)^{\frac{1}{2}}\:=\:\sum_{n=0}^\infty x^n\begin{bmatrix} \frac{1}{2}\\n\end{bmatrix}\:.
\end{equation*}
By Cauchy product expansion we obtain the generating series for $a_1(m,\frac{1}{2})$
\begin{equation*}
 \sqrt{\frac{1+x}{1-x}}\:=\:\sum_{m=0}^\infty x^m\sum_{\nu=0}^m\begin{bmatrix} \frac{1}{2}\\\nu\end{bmatrix}\begin{bmatrix} m-\nu-\frac{1}{2}\\m-\nu\end{bmatrix}
 \:=\:\sum_{m=0}^\infty x^m a_1(m,\frac{1}{2})\:.
\end{equation*}
On the other hand, by the  Taylor series above
\begin{equation*}
 \frac{1}{\sqrt{1-x^2}}\:=\:\sum_{n=0}^\infty x^{2n}\begin{bmatrix} n-\frac{1}{2}\\n\end{bmatrix}\:
\end{equation*}
we obtain the expansion
\begin{equation*}
 \sqrt{\frac{1+x}{1-x}}\:=\:\frac{1+x}{\sqrt{{1-x^2}}}\:=\:\sum_{n=0}^\infty\bigl( x^{2n}+x^{2n+1}\bigr)\begin{bmatrix} n-\frac{1}{2}\\n\end{bmatrix}\:.
\end{equation*}
Comparing coefficients, the lemma is proved.
\end{proof}
\begin{prop}\label{prop_x=1/2_values_polynomials}
\begin{itemize}
 \item [(a)]
For the polynomials $A_1(m,x)$ and $A_2(x,m)$ of Theorem~\ref{thm_summation_operator} the values in $x=\frac{1}{2}$ are given by
\begin{equation*}\label{A_1_in_1/2}
 A_1(2n,\frac{1}{2})\:=\: A_1(2n+1,\frac{1}{2})\:=\: \begin{bmatrix} n-\frac{1}{2}\\n\end{bmatrix}^2\:,
\end{equation*}
respectively
\begin{equation*}
 A_2(\frac{1}{2},m)\:=\:(2m)^2A_1(m,\frac{1}{2})\:.
\end{equation*}
\item [(b)]
For all $n\in\NN_0$ define the rational number
 \begin{equation*}
  r(n)\:=\:\sum_{k=0}^n\begin{bmatrix}k-\frac{1}{2}\\k\end{bmatrix}^2\:=\:\sum_{k=0}^n\left(\frac{1}{2^{2k}}\binom{2k}{k}\right)^2\:.
 \end{equation*}
Then the values in $x=\frac{1}{2}$ of the  polynomials $P(m,x)$ of Proposition~\ref{definierende_eigenschaften} are given by the recursion formula
\begin{equation}\label{P_in_1/2}
 P(2n+1,\frac{1}{2})\:=\:r(n)\:=\:(-1)\cdot P(2n+2,\frac{1}{2})\:.
\end{equation}
\item [(c)]
The values in $x=\frac{1}{2}$ of the polynomials $Q(m,x)$ are given by $Q(2n,\frac{1}{2})=0$, respectively
\begin{equation*}
 Q(2n+1,\frac{1}{2})\:=\:2\begin{bmatrix} n-\frac{1}{2}\\n\end{bmatrix}^2\:.
\end{equation*}
\end{itemize}
\end{prop}
\begin{proof}[Proof of Proposition~\ref{prop_x=1/2_values_polynomials}]
 Because $A_1(m,x)=a_1(m,x)^2$ the first identity of part~(a) follows from Lemma~\ref{lemma_a1(m,1/2)}.
 For the second identity recall from Proposition~\ref{properties_A1_A2}~(a) that $m^2A_1(m,x)=x^2A_2(x,m)$.
 For part~(b) we proceed by induction. By the list of $P(m,x)$ following Proposition~\ref{definierende_eigenschaften} it holds 
 $P(1,\frac{1}{2})=1=r(0)=-P(2,\frac{1}{2})$ and $P(3,\frac{1}{2})=\frac{5}{4}=r(1)=P(4,\frac{1}{2})$.
 Assume formula (\ref{P_in_1/2}) holds true for all $n<N$. 
 Then by  identity (\ref{summation_equation_to_show}) of Theorem~\ref{thm_summation_operator}
 \begin{equation*}
  P(2N+1,\frac{1}{2})+2P(2N,\frac{1}{2})+P(2N-1,\frac{1}{2})\:=\: A_1(2N,\frac{1}{2})\:,
 \end{equation*}
it follows
\begin{equation*}
 P(2N+1,\frac{1}{2})\:=\:\begin{bmatrix} N-\frac{1}{2}\\N\end{bmatrix}+2r(N-1)-r(N-1)\:=\:r(N)\:.
 \end{equation*}
Similarly, from
\begin{equation*}
  P(2N+2,\frac{1}{2})+2P(2N+1,\frac{1}{2})+P(2N,\frac{1}{2})\:=\: A_1(2N+1,\frac{1}{2})
 \end{equation*}
 we deduce
 \begin{equation*}
  P(2N+2,\frac{1}{2})\:=\:\begin{bmatrix} N-\frac{1}{2}\\N\end{bmatrix}-2r(N)+r(N-1)\:=\:-r(N)\:.
 \end{equation*}
Concerning part~(c), by (\ref{Q-S-equation}) it holds
\begin{equation*}
 Q(m,\frac{1}{2})\:=\:P(m,\frac{1}{2})+P(m,-\frac{1}{2})+P(m-1,\frac{1}{2})+P(m-1,-\frac{1}{2})\:.
\end{equation*}
Recalling the polynomials $P(m,x)$ are even functions, we obtain
\begin{equation*}
 Q(m,\frac{1}{2})\:=\:2\bigl( P(m,\frac{1}{2})+P(m-1,\frac{1}{2})\bigr)\:.
\end{equation*}
By part (b), this is zero in case $m=2n$ is even, whereas in case $m=2n+1$ we obtain $Q(2n+1,\frac{1}{2})=2(r(n)-r(n-1))=2\begin{bmatrix} n-\frac{1}{2}\\n\end{bmatrix}^2$.
\end{proof}
By the set of initial values $P(m,\frac{1}{2})=P(m,-\frac{1}{2})$, $A_1(m,\frac{1}{2})$, and $A_2(\frac{1}{2},m)$ for all $m\in\NN$, we obtain a recursion for the values 
$P(m,\frac{2k+1}{2})$, $A_1(m,\frac{2k+1}{2})$, and $A_2(\frac{2k+1}{2},m)$ for all $k\in\NN$ as follows.
Identity (\ref{summation_equation_2_to_show}) implies for all $m\in\NN$
\begin{equation*}
 P(m,\frac{2k+1}{2})\:= A_2(\frac{2k-1}{2},m)-2P(m,\frac{2k-1}{2})-P(m,\frac{2k-3}{2})\:.
\end{equation*}
Inserting this into identity (\ref{summation_equation_to_show}) yields for all $m\in\NN$
\begin{equation*}
 A_1(m,\frac{2k+1}{2})\:=\:P(m+1,\frac{2k+1}{2})+2P(m,\frac{2k+1}{2})+P(m-1,\frac{2k+1}{2})\:.
\end{equation*}
Then by Proposition~\ref{properties_A1_A2}~(a)
\begin{equation*}
 A_2(\frac{2k+1}{2},m)\:=\:\bigl(\frac{2m}{2k+1}\bigr)^2A_1(m,\frac{2k+1}{2})\:,
\end{equation*}
which closes the recursion cycle. By (\ref{Q-S-equation}), we obtain the values $Q(m,\frac{2k+1}{2})$ as well.
\bigskip

We remark another polynomial identity which arises from the above considerations for the polynomials $Q(m,x)$.
We rearrange
\begin{align*}
 A_2(x,m)=&\left(\sum_{\mu=0}^m\begin{bmatrix}
                                x+m-1-\mu\\m-1
                               \end{bmatrix}\binom{m}{\mu}\right)^2\\
&=\sum_{\mu=0}^m\binom{m}{\mu}^2f(x+m-1-\mu,0,m-1)\\
                               & + 2\cdot\sum_{0\leq\nu<\mu\leq m}\binom{n}{\nu}\binom{m}{\mu}f(x+m-1-\nu,\mu-\nu,m-1)\:.
\end{align*}
Here the polynomials $f(x,\nu,m-1)$ were defined in Proposition~\ref{prop-preimages} where we determined their preimages under $\widetilde S$. Hence we may read off the following proposition.
\begin{prop}\label{prop_new_polynomial_id}
 For the polynomials $A_2(x,m)$ defined in Theorem~\ref{thm_summation_operator} it holds true
 \begin{align*}
  \widetilde E\bigl(A_2(x,m)\bigr)=&
  \sum_{\mu=0}^{m}\binom{m}{\mu}^2F(x+m-1-\mu,0,m-1)\\
  &+ 2\cdot\sum_{0\leq\nu<\mu\leq m}\binom{m}{\nu}\binom{m}{\mu}F(x+m-1-\nu,\mu-\nu,m-1)\:,
 \end{align*}
 where the polynomials $F(x,\nu,m-1)$ were defined in Proposition~\ref{prop-preimages}.
\end{prop}
\begin{cor}
 There is a non-trivial polynomial identity given by $P(m,x)+P(m,x-1)=\widetilde E\bigl(A_2(x,m)\bigr)$,
 \begin{eqnarray*}
  &&\sum_{\nu,\mu=0}^{m-1}\frac{t(\nu,\mu,m;x)t(\mu^\ast,\nu^\ast,m;x)+t(\nu,\mu,m;x-1)t(\mu^\ast,\nu^\ast,m;x-1)}{\nu!\nu^\ast!\mu!\mu^\ast!}\\
  &&\quad =\quad \sum_{\mu=0}^{m}\binom{m}{\mu}^2F(x+m-1-\mu,0,m-1)\\
  &&\quad\quad\quad+ 2\cdot\sum_{0\leq\nu<\mu\leq m}\binom{m}{\nu}\binom{m}{\mu}F(x+m-1-\nu,\mu-\nu,m-1)\:.
 \end{eqnarray*}
\end{cor}



\end{document}